\documentclass{amsart}

\usepackage{amsfonts, amssymb, amsmath, eucal, verbatim, amsthm, amscd, enumerate}
\usepackage{graphicx}

\newtheorem{theorem}{Theorem}[section]
\newtheorem{lemma}[theorem]{Lemma}
\newtheorem{proposition}[theorem]{Proposition}

\newtheorem{corollary}[theorem]{Corollary}

\theoremstyle{definition}

\theoremstyle{remark}
\newtheorem*{remark}{Remark}

\newtheorem*{acknowledgement}{Acknowledgement}

\title{A sharp Sobolev--Strichartz estimate for the wave equation}
\author{Neal Bez}
\address{Neal Bez, Department of Mathematics, Graduate School of Science and Engineering,
Saitama University, Saitama 338-8570, Japan}
\email{nealbez@mail.saitama-u.ac.jp}

\author{Chris Jeavons}
\address{Chris Jeavons, School of Mathematics, University of Birmingham, Edgbaston, Birmingham, B15 2TT, UK}
\email{jeavonsc@maths.bham.ac.uk}

\date{\today}

\begin{document}

\begin{abstract}
We calculate the the sharp constant and characterise the extremal initial data in $\dot{H}^{\frac{3}{4}} \times \dot{H}^{-\frac{1}{4}}$ for the $L^4$ Sobolev--Strichartz estimate for the wave equation in four space dimensions.
\end{abstract}
\maketitle
\section{Introduction}
For $d\geq 2$ and $s\in \left[\frac{1}{2},\frac{d}{2}\right)$ the well-known Sobolev--Strichartz estimate for the one-sided wave propagator states that, for some finite constant $C > 0$, 
\begin{equation}\label{WSSt}
\|e^{it\sqrt{-\Delta}}f\|_{L^p(\mathbb{R}^{d+1})} \leq C\|f\|_{\dot{H}^s(\mathbb{R}^d)}
\end{equation}
for each $f$ in the homogeneous Sobolev space $\dot{H}^s(\mathbb{R}^d)$, with norm given by $\|f\|_{\dot{H}^s(\mathbb{R}^d)}=\|(-\Delta)^{\frac{s}{2}}f\|_{L^2(\mathbb{R}^d)}$, and where
\[p=\frac{2(d+1)}{d-2s}.\]

The sharp constant in the estimate \eqref{WSSt} given by
\[
\mathbf{W}(d,s):=\operatorname*{sup}_{f\in\dot{H}^s \setminus \{0\}}\frac{\|e^{it\sqrt{-\Delta}}f\|_{L^p(\mathbb{R}^{d+1})}}{\|f\|_{\dot{H}^s(\mathbb{R}^d)}}
\]
has attracted attention in recent years; however, to date, the value of $\mathbf{W}(d,s)$ and a full characterisation of extremisers (those $f$ which attain the supremum) has been established only in some rather isolated cases. It is known that, for all admissible $(d,s)$, an extremiser \emph{exists} (see \cite{Bulut}, \cite{FVV}, \cite{Ramos}). Identifying the \emph{exact shape} of such extremisers appears to be a rather difficult problem, with prior results in this direction only available in the cases $(d,s)$ equal to $(2,\frac{1}{2})$ and $(3,\frac{1}{2})$, due to Foschi \cite{Foschi}, and the case $(d,s)$ equal to $(5,1)$ in \cite{BR}. In each of these cases, the initial datum $f_\star$ whose Fourier transform is given by
\begin{equation*} 
\widehat{f_\star}(\xi) = \frac{e^{-|\xi|}}{|\xi|}
\end{equation*}
is extremal; in fact, these works also gave a full characterisation of the extremal data by showing that any extremiser $f$ coincides with $f_\star$ up to the action of a certain group of transformations (which are slightly different when $s=\frac{1}{2}$ and $s = 1$). Based on these results, it is tempting to boldly conjecture that such $f$ are extremisers for all admissible $(d,s)$. Whilst this is premature, the purpose of this short paper is to add further weight and show that this is indeed the case for $(d,s) = (4,\frac{3}{4})$.
\begin{theorem}\label{thm1}
The one-sided wave propagator satisfies the estimate
\begin{equation}\label{eq1}
\|e^{it\sqrt{-\Delta}}f\|_{L^4(\mathbb{R}^5)} \leq \mathbf{W}(4,\tfrac{3}{4})\left\|f\right\|_{\dot{H}^{\frac{3}{4}}(\mathbb{R}^4)}
\end{equation}
with constant
\[\mathbf{W}(4,\tfrac{3}{4})=\left(\frac{4}{15\pi^2}\right)^{\frac{1}{4}}.\]
The constant is sharp and is attained if and only if 
\[\widehat{f}(\xi)=\frac{e^{a|\xi|+ib\cdot\xi+c}}{|\xi|},\]
where $a, c\in\mathbb{C}$ such that $\operatorname{Re}(a)<0$, and $b\in \mathbb{R}^d$.
\end{theorem}
Our proof of Theorem \ref{thm1} relies on a sharp estimate for the one-sided wave propagator from \cite{BR}; this is followed by a further argument using spherical harmonics inspired by recent work of Foschi \cite{Foschi2} on the sharp Stein--Tomas adjoint Fourier restriction theorem for the sphere $\mathbb{S}^2$ in $\mathbb{R}^3$. We also show that such an approach may be used to recover in a new manner the characterisation of extremisers in \cite{BR} for the case $(d,s) = (5,1)$. 

For the full solution of the wave equation, we may deduce the following sharp Sobolev--Strichartz estimate and characterisation of extremal initial data.
\begin{corollary} \label{cor}
The solution of the wave equation $\partial_{tt}u=\Delta u$ on $\mathbb{R}^4 \times \mathbb{R}$ with initial data $(u(0),\partial_tu(0))$ satisfies
\[
\|u\|_{L^4(\mathbb{R}^5)}\leq \left(\frac{1}{10\pi^2}\right)^{\frac{1}{4}}\left(\|u(0)\|_{\dot{H}^{\frac{3}{4}}(\mathbb{R}^4)}^2+\|\partial_t u(0)\|_{\dot{H}^{-\frac{1}{4}}(\mathbb{R}^4)}^2\right)^\frac{1}{2}
\]
and the constant is sharp. Furthermore, the initial data given by
\[
(u(0),\partial_t u(0))=(0, (1+|x|^2)^{-\frac{5}{2}}),
\]
is extremal and generates the set of all extremal initial data under the action of the group generated by the transformations: 
\vspace{2mm}
\newline \textbullet \,\, space-time translations $u(t,x)\to u(t+t_0,x+x_0)$ with $(t_0,x_0) \in \mathbb{R}^{d+1}$;
\newline \textbullet \,\, parabolic dilations $u(t,x)\to u(\mu^2 t, \mu x)$ with $\mu > 0$;
\newline \textbullet \,\, change of scale $u(t,x) \to \mu u(t,x)$ with $\mu > 0$;
\newline \textbullet \,\, phase shift $u(t,x) \to e^{i\theta} u(t,x)$ with $\theta\in\mathbb{R}$.
\end{corollary}

Our results here also fit into a broader collection of recent papers on sharp Sobolev--Strichartz estimates for dispersive propagators, where, broadly speaking, the question of existence of extremisers is well-understood yet the identification of their shape has only been established in rather special cases (see, for example, \cite{FVV}, \cite{Foschi}, \cite{Jeavons},  \cite{Quilodran}, \cite{Shao}).

In Section 2 we prove Theorem \ref{thm1} and Corollary \ref{cor}, and in Section 3 we adapt our method to obtain an alternative proof of the analogous result from \cite{BR} for the case $(d,s)=(5,1)$.

\begin{acknowledgement}
The authors express their thanks to Jon Bennett for helpful conversations.
\end{acknowledgement}

\section{Proof of Theorem \ref{thm1} and Corollary \ref{cor}}
A key ingredient in the proof of Theorem \ref{thm1} is the following sharp inequality proved in \cite{BR}. Here we use the notation $y^\prime=\frac{y}{|y|}$, for $y\in\mathbb{R}^d \setminus \{0\}$. 
\begin{theorem} \label{t:BR}
Let $d\geq 3$. Then
\begin{align}\label{WV}
\|e^{it\sqrt{-\Delta}} f \|_{L^4(\mathbb{R}^{d+1})}^4  \leq \mathbf{C}(d)\int_{\mathbb{R}^{2d}}|\widehat{f}(y_1)|^2|\widehat{f}(y_2)|^2 |y_1|^{\frac{d-1}{2}}|y_2|^{\frac{d-1}{2}}(1-y_1^\prime\cdot y_2^\prime)^{\frac{d-3}{2}}\,\mathrm{d}y_1\mathrm{d}y_2
\end{align}
holds with sharp constant
\[
\mathbf{C}(d)=2^{-\frac{d-1}{2}}(2\pi)^{-3d+1}|\mathbb{S}^{d-1}|
\]
which is attained if and only if 
\[
\widehat{f}(\xi) = \frac{e^{a|\xi|+b\cdot\xi+c}}{|\xi|},
\]
where $a,c \in\mathbb{C}$, $b\in\mathbb{C}^d$ with $|\emph{Re} (b)|<-\emph{Re}(a)$. 
\end{theorem}
The one-sided wave propagator is given by
\[
e^{it\sqrt{-\Delta}}f(x)=\frac{1}{(2\pi)^d}\int_{\mathbb{R}^d}e^{ix\cdot\xi+it|\xi|}\widehat{f}(\xi)\,\mathrm{d}\xi, \quad (x,t) \in \mathbb{R}^d \times \mathbb{R},
\]
for appropriate functions $f$, and the Fourier transform we use is
\[
\widehat{f}(\xi)=\int_{\mathbb{R}^d}f(x)e^{-ix\cdot\xi}\,\mathrm{d}x, \quad \xi\in\mathbb{R}^d.
\]

Our observation is that if we introduce polar coordinates for $y_1$ and $y_2$ in \eqref{WV}, then we are led to real-valued functionals of the form
\begin{equation*}
H_{\lambda}(g)=\int_{\mathbb{S}^{d-1}\times \mathbb{S}^{d-1}}g(\eta_1)\overline{g(\eta_2)}|\eta_1-\eta_2|^{-\lambda}\,\mathrm{d}\eta_1\mathrm{d}\eta_2
\end{equation*}
for $g\in L^1(\mathbb{S}^{d-1})$ and $\lambda \leq 0$. This is reminiscent of recent work of Foschi \cite{Foschi2} where a sharp upper bound for $H_{-1}$ was established for $d=3$. For Theorem \ref{thm1} we need an analogous result for $d = 4$; this is contained in the subsequent proposition, which we state more generally to highlight why our approach only works as it stands for $d=4,5$.

First, we introduce the beta function
\[
\mathrm{B}(x,y)=\int_0^1 t^{x-1}(1-t)^{y-1}\,\mathrm{d}t, \quad x,y>0,
\]
and $\mu_g$ to denote the average value of $g$ on the sphere. Also, we use $\mathbf{1}$ for the function which is identically equal to one on the sphere.
\begin{proposition}\label{claim1}
Let $-2<\lambda < 0$, and let $g$ be any $L^1$ function on $\mathbb{S}^{d-1}$. Then, 
\[
H_\lambda(g)\leq H_\lambda(\mu_g \mathbf{1})=  2^{d-2-\lambda}\mathrm{B}(\tfrac{d-1-\lambda}{2},\tfrac{d-1}{2})\frac{|\mathbb{S}^{d-2}|}{\left|\mathbb{S}^{d-1}\right|}\left|\int_{\mathbb{S}^{d-1}}g\right|^2.
\]
Further, equality holds if and only if $g$ is constant.
\end{proposition} 
Following Foschi \cite{Foschi2}, our proof of Proposition \ref{claim1} is based on a spectral argument using a spherical harmonic decomposition of $g$ and the Funk--Hecke formula to obtain explicit expressions for the eigenvalues. We remark that similar types of arguments have proved profitable in understanding sharp forms of other important estimates; see, for example, \cite{Beckner}, \cite{BezSugimoto} and \cite{FrankLieb}. The connection to the latter paper deserves a further remark; indeed, in \cite{FrankLieb}, Frank and Lieb provide a reproof of the sharp Hardy--Littlewood--Sobolev inequality on the sphere, originally due to Lieb \cite{Lieb}, which gives the sharp upper bound on $H_\lambda$ for $0 < \lambda < d-1$ in terms of the $L^p$ norm of $g$, where $p = \frac{2(d-1)}{2(d-1)-\lambda}$.



The information we need concerning the eigenvalues is contained in the following lemma. Here we use $P_{k,d}$ to denote the Legendre polynomial of degree $k$ in $d$ dimensions, which may be defined using the generating function 
\[
\frac{1}{(1+r^2-2rt)^{\frac{d-2}{2}}}=\sum_{k=0}^\infty \binom{k+d-3}{d-3}r^k P_{k,d}(t), \quad |r|<1, |t|\leq 1.
\]
\begin{lemma}\label{lemma1}
Let $-2<\lambda < 0$, and define
\[
\mathrm{I}_k(d,\lambda) = |\mathbb{S}^{d-2}|\int_{-1}^1 (1-t)^{-\frac{\lambda}{2}}P_{k,d}(t)(1-t^2)^{\frac{d-3}{2}}\,\mathrm{d}t.
\]
Then
\[
\mathrm{I}_0(d,\lambda) = |\mathbb{S}^{d-2}|2^{d-2-\frac{\lambda}{2}}\mathrm{B}(\tfrac{d-1-\lambda}{2},\tfrac{d-1}{2}) > 0
\]
and $\mathrm{I}_k(d,\lambda) < 0$ for all $k\geq 1$. 
\end{lemma}
\begin{remark}
The inequality in Proposition \ref{claim1} is false if $\lambda<-2$. This is because $(-1)^k \mathrm{I}_k(d,\lambda) > 0$ for $k \geq 0$ up to some threshold; for example $\mathrm{I}_2(d,\lambda)>0$ for such $\lambda$. This is the reason why our approach does not allow us to prove a generalisation of Theorem \ref{thm1} to dimension 6 and above (for general $d$, we should take $\lambda = 3-d$). A similar obstacle arises in \cite{CO} when generalising Foschi's argument to obtain the result in \cite{Foschi2} in higher dimensions. At the endpoint $\lambda=-2$ the sharp inequality in Proposition \ref{claim1} still holds but one also has equality for certain non-constant functions $g$. This turns out not to matter for our application and we can recover the sharp inequality and characterisation of extremisers for \eqref{WSSt} in the case $(d,s)=(5,1)$ first proved in \cite{BR}; we expand upon this point in Section 3.
\end{remark}

Assume Lemma \ref{lemma1} to be true for the moment, then to prove Proposition \ref{claim1}, we first observe that it suffices by density and continuity of the functional $H_\lambda$ on $L^1(\mathbb{S}^{d-1})$ to consider $g\in L^2(\mathbb{S}^{d-1})$. We may then write $g=\sum_{k\geq 0} Y_k$ as a sum of orthogonal spherical harmonics; upon which it follows that
\begin{equation}\label{int1}
H_\lambda(g)=2^{-\frac{\lambda}{2}}\sum_{k\geq 0}\int_{\mathbb{S}^{d-1}}g(\eta_1)\int_{\mathbb{S}^{d-1}}\overline{Y_k(\eta_2)} (1-\eta_1\cdot\eta_2)^{-\frac{\lambda}{2}}\,\mathrm{d}\eta_2\mathrm{d}\eta_1.
\end{equation}
To deal with the inner integral in \eqref{int1} we use the Funk--Hecke formula for the spherical harmonics, which states that
\begin{equation*} 
\int_{\mathbb{S}^{d-1}}Y_k(\eta)F(\omega\cdot\eta)\,\mathrm{d}\eta=\Lambda_kY_k(\omega)
\end{equation*}
for $\omega\in\mathbb{S}^{d-1}$ and $k\in\mathbb{N}_0$, where
\[
\Lambda_k:=|\mathbb{S}^{d-2}|\int_{-1}^1F(t)P_{k,d}(t)(1-t^2)^{\frac{d-3}{2}}\,\mathrm{d}t,
\]
provided that $F \in L^1([-1,1],(1-t^2)^{\frac{d-3}{2}})$ (see \cite{AH}, pp. 35--36). It then follows that the inner integral in \eqref{int1} evaluates to a (positive) constant multiple of $\mathrm{I}_k(d,\lambda) \overline{Y_k(\eta_1)}$. Precisely, using the orthogonality of the spherical harmonics of different degrees and Lemma \ref{lemma1},
\begin{equation*}
H_{\lambda}(g)=2^{-\frac{\lambda}{2}}\sum_{k\geq 0}\mathrm{I}_k(d,\lambda) \int_{\mathbb{S}^{d-1}} \left|Y_{k}(\eta)\right|^2\,\mathrm{d}\eta \leq 2^{-\frac{\lambda}{2}}\mathrm{I}_0(d,\lambda) \int_{\mathbb{S}^{d-1}} \left|Y_{0}\right|^2\,\mathrm{d}\eta =H_{\lambda}(\mu_g \mathbf{1}).
\end{equation*}
Equality is clearly satisfied for $g = Y_0$ or equivalently $g$ which are constant. There are no further cases of equality since $\mathrm{I}_k(d,\lambda)$ is strictly negative for $k \geq 1$, by Lemma \ref{lemma1}. 

Using the expression for $\mathrm{I}_0(d,\lambda)$ in Lemma \ref{lemma1} and the definition of $\mu_g$, it is then easy to derive the claimed expression for $H_\lambda(\mu_g\mathbf{1})$, which completes the proof of Proposition \ref{claim1}.

\proof[Proof of Lemma \ref{lemma1}] By a simple change of variables, it is easily checked that $\mathrm{I}_0(d,\lambda)$ satisfies the claimed equality in terms of the beta function. To prove the strict negativity of $\mathrm{I}_k(d,\lambda)$ for $k \geq 1$, we first use the Rodrigues formula for $P_{k,d}$ (see \cite{AH}, pp. 37), which states
\begin{equation*}
(1-t^2)^{\frac{d-3}{2}}P_{k,d}(t)=(-1)^k R_{k,d}\frac{\mathrm{d}^k}{\mathrm{d}t^k}(1-t^2)^{k+\frac{d-3}{2}}, \quad t\in \left[-1,1\right],
\end{equation*}
with
\[R_{k,d}=\frac{\Gamma(\frac{d-1}{2})}{2^k\Gamma(k+\frac{d-1}{2})}>0,\]
to obtain that
\[\mathrm{I}_k(d,\lambda) =(-1)^k R_{k,d} \int_{-1}^1 (1-t)^{-\frac{\lambda}{2}}\frac{\mathrm{d}^k}{\mathrm{d}t^k}(1-t^2)^{k+\frac{d-3}{2}}\,\mathrm{d}t.\]
Integrating by parts, the boundary terms disappear and we obtain
\begin{equation}\label{ibp-eq}
\mathrm{I}_k(d,\lambda) = (-1)^{k} R_{k,d}\left(-\frac{\lambda}{2}\right)\int_{-1}^1 (1-t)^{-\frac{\lambda}{2}-1}\frac{\mathrm{d}^{k-1}}{\mathrm{d}t^{k-1}}(1-t^2)^{k+\frac{d-3}{2}}\,\mathrm{d}t.
\end{equation}
Since $-\frac{\lambda}{2}>0$, the sign of the constant in front of the integral in \eqref{ibp-eq} does not change at the first integration by parts. However, since $-\frac{\lambda}{2}-1<0$, at every integration by parts step after the first, we will incur a sign change. Hence, integrating by parts a total of $k$ times, we see that $\mathrm{I}_k(d,\lambda)$ evaluates to
\begin{align*}
-C_k(d,\lambda) \int_{-1}^1 (1-t)^{-\frac{\lambda}{2}-k}(1-t^2)^{k+\frac{d-3}{2}}\,\mathrm{d}t
\end{align*}
for some strictly positive constant $C_k(d,\lambda)$. Hence $\mathrm{I}_k(d,\lambda) < 0$ as claimed.
\endproof
\proof[Proof of Theorem \ref{thm1}] If we set $d=4$ and write the integral on the right-hand side of \eqref{WV} using polar coordinates, we get
\begin{align*}
\int_{(0,\infty)^2} \int_{(\mathbb{S}^3)^2} |\widehat{f}(r_1\eta_1)|^2|\widehat{f}(r_2\eta_2)|^2 r_1^{\frac{9}{2}}r_2^{\frac{9}{2}}(1-\eta_1\cdot\eta_2)^{\frac{1}{2}} \,\mathrm{d}\eta\mathrm{d}r = \frac{1}{\sqrt{2}} \int_{\mathbb{S}^3\times\mathbb{S}^3}g(\eta_1)g(\eta_2)|\eta_1-\eta_2|\,\mathrm{d}\eta_1\mathrm{d}\eta_2,
\end{align*}
where
\begin{equation}\label{g-def}
g(\eta):=\int_0^\infty |\widehat{f}(r\eta)|^2r^{\frac{9}{2}}\,\mathrm{d}r,
\end{equation}
for $\eta \in \mathbb{S}^3$. By Plancherel's theorem, 
\[
\int_{\mathbb{S}^3} g  = (2\pi)^4\|f\|_{\dot{H}^{\frac{3}{4}}(\mathbb{R}^4)}^2.
\]
If we then apply \eqref{WV} and take $\lambda=-1$ in Proposition \ref{claim1}, we have
\begin{equation}\label{eqchain}
\|e^{it\sqrt{-\Delta}}f\|_{L^4(\mathbb{R}^5)}^4\leq \frac{\mathbf{C}(4)}{\sqrt{2}} H_{-1}(g) \leq \frac{\mathbf{C}(4)}{\sqrt{2}} H_{-1}(\mathbf{1}) |\mu_g|^2 = \frac{4}{15\pi^2} \|f\|_{\dot{H}^{\frac{3}{4}}(\mathbb{R}^4)}^4,
\end{equation}
as claimed. The first inequality in \eqref{eqchain} is an equality when $f$ extremises inequality \eqref{WV}, and the second is an equality when the function $g$ defined by \eqref{g-def} is constant on $\mathbb{S}^3$. In particular, equality holds in both cases for $f$ given by 
\[
\widehat{f}(\xi)=\frac{e^{a|\xi|+ib\cdot\xi+c}}{|\xi|},
\]
where $a, c\in\mathbb{C}$ such that $\text{Re}(a)<0$, and $b\in \mathbb{R}^4$. Note that, for such $f$, we have that $|\widehat{f}|$ is radial (and hence $g$ is constant).

On the other hand, if $f$ is an extremiser for \eqref{eq1}, then we must have equality at both of the inequalities in \eqref{eqchain}. From the first inequality, using Theorem \ref{t:BR}, we see that necessarily
\[
\widehat{f}(\xi) = \frac{e^{a|\xi| + b \cdot \xi + c}}{|\xi|}\,,
\]
where $a, c \in \mathbb{C}$, $b \in \mathbb{C}^4$ and $\mbox{Re}(a) < - |\mbox{Re}(b)|$. However, in this case,
\[
g(\eta) = e^{2\text{Re}(c)} \int_0^\infty e^{2r(\text{Re}(a) + \text{Re}(b) \cdot \eta)} r^\frac{5}{2} \, \mathrm{d}r = \frac{e^{2\text{Re}(c)} \int_0^\infty e^{-r} r^\frac{5}{2} \, \mathrm{d}r}{[-2(\text{Re}(a) + \text{Re}(b) \cdot \eta)]^{\frac{7}{2}}}
\]
and for this to be constant in $\eta$, we must have $\mbox{Re}(b) = 0$. This completes the proof of Theorem \ref{thm1}.
\endproof

\begin{proof}[Proof of Corollary \ref{cor}]
Write the solution of the wave equation $u$ as $e^{it\sqrt{-\Delta}}f_+ + e^{-it\sqrt{-\Delta}}f_-$, where the functions $f_+$ and $f_-$ are defined using the initial data by
\[
u(0)=f_++f_-, \quad \partial_t u(0)=i\sqrt{-\Delta}(f_+-f_-).
\]
Using orthogonality and the Cauchy--Schwarz inequality on $L^2(\mathbb{R}^5)$, we get
\begin{align*}
\|u\|_{L^4(\mathbb{R}^5)}^4 & = \|e^{it\sqrt{-\Delta}} f_+ \|_{L^4(\mathbb{R}^5)}^4 + \|e^{-it\sqrt{-\Delta}} f_- \|_{L^4(\mathbb{R}^5)}^4 + 4\|e^{it\sqrt{-\Delta}} f_+ e^{-it\sqrt{-\Delta}} f_-\|_{L^2(\mathbb{R}^5)}^2 \\
& \leq \|e^{it\sqrt{-\Delta}} f_+ \|_{L^4(\mathbb{R}^5)}^4 + \|e^{-it\sqrt{-\Delta}} f_- \|_{L^4(\mathbb{R}^5)}^4 + 4\|e^{it\sqrt{-\Delta}} f_+ \|_{L^4(\mathbb{R}^5)}^2\|e^{-it\sqrt{-\Delta}} f_- \|_{L^4(\mathbb{R}^5)}^2.
\end{align*}
The basic inequality $2(X^2 + Y^2 + 4XY) \leq 3(X+Y)^2$ and Theorem \ref{thm1}, which clearly also holds for $e^{-it\sqrt{-\Delta}}$, now yield
\[
\|u\|_{L^4(\mathbb{R}^5)}^4 \leq \frac{3}{8} \mathbf{W}(4,\tfrac{3}{4})^4 \left(\|u(0)\|_{\dot{H}^{\frac{3}{4}}(\mathbb{R}^4)}^2+\|\partial_t u(0)\|_{\dot{H}^{-\frac{1}{4}}(\mathbb{R}^4)}^2\right)^2
\]
which gives the claimed inequality in Corollary \ref{cor}. 

The above argument was used by Foschi in \cite{Foschi} when $(d,s) = (3,\frac{1}{2})$ and in \cite{BR} when $(d,s) = (5,1)$. The characterisation of extremisers also follows in the analogous way, and so we refer the reader to \cite{Foschi} or \cite{BR} and omit the details. 
\end{proof}

\section{Five spatial dimensions}
We conclude by presenting an alternative derivation of the sharp constant and characterisation of extremisers for the estimate \eqref{WSSt} in the case $(d,s)=(5,1)$, in the spirit of the argument in the previous section. 

For this, we need an appropriate modification of Proposition \ref{claim1} and thus Lemma \ref{lemma1} for $d=5$ and $\lambda = -2$. However, it is straightforward to see that
\[
\mathrm{I}_k(5,-2) = |\mathbb{S}^{3}|\int_{-1}^1 (1-t) P_{k,5}(t)(1-t^2) \,\mathrm{d}t
\]
satisfies $\mathrm{I}_0(5,-2) > 0$, $\mathrm{I}_1(5,-2) < 0$ and $\mathrm{I}_k(5,-2)$ vanishes for all $k \geq 2$. Thus
\[
H_{-2}(g) = \frac{\mathrm{I}_0(5,-2) }{2} \|Y_0\|_{L^2(\mathbb{S}^{4})}^2 + \frac{\mathrm{I}_1(5,-2)}{2} \|Y_1\|_{L^2(\mathbb{S}^{4})}^2 \leq H_{-2}(\mathbf{1})|\mu_g|^2,
\]
where $g =\sum_{k\geq 0} Y_k$ is the expansion of $g$ into spherical harmonics. Here, equality holds if $g$ is constant, but unlike the estimates in Proposition \ref{claim1}, there are further cases of equality.

Taking $f \in \dot{H}^1(\mathbb{R}^5)$ and applying this with $g$ given by
\begin{equation}\label{G-def2}
g(\eta) :=\int_0^\infty |\widehat{f}(r\eta)|^2r^{6}\,\mathrm{d}r
\end{equation}
for $\eta\in\mathbb{S}^4$, we have 
\begin{equation*}
\|e^{it\sqrt{-\Delta}}f\|_{L^4(\mathbb{R}^6)}^4\leq \frac{\mathbf{C}(5)}{2}H_{-2}(g) \leq  \frac{\mathbf{C}(5)}{2}H_{-2}(\mathbf{1}) |\mu_g|^2   = \frac{1}{24\pi^2} \|f\|_{\dot{H}^{1}(\mathbb{R}^5)}^4.
\end{equation*}
As before, equality holds in both inequalities for $f$ given by 
\[
\widehat{f}(\xi)=\frac{e^{a|\xi|+ib\cdot\xi+c}}{|\xi|},
\]
where $a, c\in\mathbb{C}$ such that $\text{Re}(a)<0$, and $b\in \mathbb{R}^5$. 

Conversely, if $f$ is an extremiser, then Theorem \ref{t:BR} implies that
\begin{equation*}
\widehat{f}(\xi) = \frac{e^{a|\xi| + b \cdot \xi + c}}{|\xi|}\,,
\end{equation*}
where $a, c \in \mathbb{C}$, $b \in \mathbb{C}^5$ and $\text{Re}(a) < - |\text{Re}(b)|$. Substituting our function $f$ into \eqref{G-def2}, we see that it suffices to consider 
\begin{align*}
g(\eta) = e^{2\text{Re}(c)}\int_0^\infty e^{2r(\text{Re}(a)+\text{Re}(b)\cdot\eta)}r^4\,\mathrm{d}r = \frac{e^{2\text{Re}(c)}}{32(-\text{Re}(a)-\text{Re}(b) \cdot\eta)^5}\int_0^\infty e^{-r}r^4\,\mathrm{d}r.
\end{align*}
Since $\mathrm{I}_1(5,-2) < 0$, we must have that $\|Y_1\|_{L^2}=0$. On the other hand, using the projection $\Pi$ onto the space of spherical harmonics of degree one given by
\[
\Pi g(\eta) \mapsto \frac{5}{|\mathbb{S}^4|} \int_{\mathbb{S}^4} P_{1,5}(\eta \cdot \omega) g(\omega) \, \mathrm{d}\omega
\]
for each $\eta \in \mathbb{S}^4$, it follows that
\begin{equation} \label{e:Y1}
Y_1(\eta) = C \int_{\mathbb{S}^4}  \frac{P_{1,5}(\eta \cdot \omega)}{(-\text{Re}(a)-\text{Re}(b) \cdot\eta)^5} \, \mathrm{d}\omega
\end{equation}
for some absolute constant $C > 0$. If we suppose, for a contradiction, that $\text{Re}(b) \neq 0$, then an application of the Funk--Hecke formula implies that 
\begin{equation} \label{e:Y1'}
Y_1(\eta) = C P_{1,5}(\eta \cdot \text{Re}(b)') \int_{-1}^1 \frac{t(1-t^2)}{(1+At)^5}\,\mathrm{d}t
\end{equation}
for each $\eta \in \mathbb{S}^4$, where $A:=\frac{|\text{Re}(b)|}{\text{Re}(a)}\in \left(-1,0\right]$. The absolute constants $C > 0$ in \eqref{e:Y1} and \eqref{e:Y1'} may not be the same. Since $Y_1$ vanishes almost everywhere on $\mathbb{S}^4$, it follows that the integral on the right-hand side of \eqref{e:Y1'} vanishes. This forces $A = 0$, which gives the desired contradiction. 

The above argument provides an alternative proof of the following, and at the level of the proof, unifies it with Theorem \ref{thm1}.
\begin{theorem}[\cite{BR}, Corollary 2.2]
The one-sided wave propagator satisfies the estimate
\begin{equation*}
\|e^{it\sqrt{-\Delta}}f\|_{L^4(\mathbb{R}^6)} \leq \mathbf{W}(5,1)\left\|f\right\|_{\dot{H}^{1}(\mathbb{R}^5)}
\end{equation*}
with constant
\[\mathbf{W}(5,1)=\left(\frac{1}{24\pi^2}\right)^{\frac{1}{4}}.\]
The constant is sharp and is attained if and only if 
\[\widehat{f}(\xi)=\frac{e^{a|\xi|+ib\cdot\xi+c}}{|\xi|},\]
where $a, c\in\mathbb{C}$ such that $\operatorname{Re}(a)<0$, and $b\in \mathbb{R}^5$.
\end{theorem}

\end{document}